\documentclass[10pt,twocolumn,letterpaper]{article}

\usepackage[pagenumbers]{cvpr} % To force page numbers, e.g. for an arXiv version

\usepackage[dvipsnames]{xcolor}

\definecolor{cvprblue}{rgb}{0.21,0.49,0.74}
\usepackage[pagebackref,breaklinks,colorlinks,citecolor=cvprblue]{hyperref}

\def\paperID{17142} % *** Enter the Paper ID here
\def\confName{CVPR}
\def\confYear{2024}

\title{Improving Autoencoder Image Interpolation via Dynamic Optimal Transport}

\author{
    Xue Feng\\
    University of California, Davis\\
    {\tt\small xffeng@ucdavis.edu}
    \and
    Thomas Strohmer\\
    University of California, Davis\\
    {\tt\small strohmer@math.ucdavis.edu}
}

\DeclareMathOperator{\Mass}{Mass}
\DeclareMathOperator{\MassLoss}{MassLoss}

\usepackage{algorithm}
\usepackage{algpseudocode}
\usepackage{comment}

\usepackage{graphicx}
\usepackage{subcaption}
\usepackage[dvipsnames]{xcolor}         % colors
\usepackage[export]{adjustbox}

\usepackage{amsthm}
\newtheorem{theorem}{Theorem}
\theoremstyle{definition}

\newcommand{\E}{\mathbb{E}}
\newcommand{\D}{\mathbb{D}}
\newcommand{\R}{\mathbb{R}}

\newcommand{\itoj}{i \rightarrow j}
\newcommand{\m}{\mathbf{m}}
\newcommand{\Q}{\mathbf{w}}
\newcommand{\w}{\mathbf{w}}

\begin{document}
\maketitle
\begin{abstract}
Autoencoders are important generative models that, among others, have the ability to interpolate image sequences.
However, interpolated images are usually not semantically meaningful.
In this paper, motivated by dynamic optimal transport,
we consider image interpolation as a mass transfer problem and propose a novel regularization term to penalize non-smooth and unrealistic changes in the interpolation result.
Specifically, we define the path energy function for each path connecting the source and target images.
The autoencoder is trained to generate the $L^2$ optimal transport geodesic path when decoding a linear interpolation of their latent codes.
With a simple extension, this model can handle complicated environments, such as allowing mass transfer between obstacles and unbalanced optimal transport. 
A key feature of the proposed method is that it is physics-driven and can generate robust and realistic interpretation results even
when only very limited training data are available.
%Furthermore, 

\end{abstract}    
\section{Introduction}
\label{sec:intro}

%(why autoencoder --> image interpolation)
Autoencoders are known to uncover the underlying structure of a dataset by learning to encode data points to lower-dimensional latent codes which can then be decoded to reconstruct the input data.
In more recent studies, evidence has emerged demonstrating the robust capabilities of autoencoders in the domain of generative modeling~\cite{kingma2014auto,advae}.
They also have been used successfully to interpolate between data points
by decoding a convex combination of their latent codes.
This is useful for many applications, such as image sequence interpolation, also known as frame interpolation.

%(state of art work to improve,  1) GAN 2) shape the latent space )
However, this process of image interpolation by decoding a convex combination of their latent codes often
leads to artifacts or yields unrealistic outcomes.
See \Cref{fig:typical} for an example.
There are two prevalent strategies to improve the quality of the interpolated result.
The first approach is to introduce a regularization loss term to penalize unrealistic results, such as an adversarial term~\cite{advae, berthelot*2018understanding,sander2022autoencoding}.
In the adversarial regularization method,
a {\em critic network} is trained to evaluate certain criteria, such as the similarity between the interpolated image and the training data, while an autoencoder is concurrently trained to fool this critic network.
A second approach consists of shaping the latent representation to follow a manifold consistent with the training images \cite{kingma2014auto,sainburg2018generative, bouchacourt2018multi,shu2018deforming,oring2020autoencoder}.
The idea is 
that the incongruities in the interpolation result are probably caused by the fact that such straightforwardly interpolated latent vectors stray from the data manifold.
So shaping the latent space may be benifical.

%(Formulate to be a mass transfer problem; the  target is to find a geodesic path.)
In this paper, we propose a
regularization term that aligns with the first strategy, using a robust physics model to enhance the interpolated images produced by decoding a linear combination of latent codes. 
While there are numerous possibilities to define some notion of {\em data sequence interpolation},
it is natural to consider the image interpolation problem as a mass transfer problem.
For example, in a grayscale image, the value of the pixels represents the mass at that location. In this context, 
image interpolation is the process of transferring the mass from the source image to the target image. 
There are infinitely many possible paths to transfer the mass. 
We define the best path as the one with the least transportation cost.

In the past decade, optimal transport (OT) has developed into a highly-regarded research field in machine learning areas such as generative modeling \cite{arjovsky17a}, domain adaption\cite{courty2014domain}, and image interpolation. The traditional optimal transport defines a family of metrics, known as the Wasserstein distance, between probability distributions. We can solve the optimal transport plan and then use it to push the source image to get the interpolation path. Nevertheless, the metric itself cannot directly assess the quality of the interpolated images.

Our motivation comes from the dynamic OT proposed by Benamou and Brenier~\cite{benamou2000}.
Dynamic OT considers the mass transfer problem in a fluid mechanics framework
and associates each path with a kinetic energy cost.
The path with minimal energy is the so-called geodesic path according to the Wasserstein metric.
The continuity equation constrain ensures that
the path will follow the law of physics.
Meanwhile,
due to its fluid mechanics nature,
dynamic OT can handle complex scenarios and model real-world problems, such as those involving obstacles or varying transport costs over time~\cite{benamou2000,papadakis2014optimal}.
%However, solving dynamic optimal transport problems can be computationally intensive, particularly for large-scale problems. 
%A regularization term can be regarded as a 
%It has been solved using some frist order solver, 
%the proximal splitting method\cite{papadakis2014optimal},
%and computationally costly.
%optimal transport aims to find the minimal cost plan toconnecting any two images
%Specially,The path energy drives the interpolated points to look reliable as it is optimized towards the geodesic path. In this work,
%we propose a path energy regularization term to improve the interpolation ability of autoencoder.

\begin{figure*}[!ht]
    \centering
    \begin{subfigure}[b]{0.9\textwidth}
        \centering
        \includegraphics[width=0.9\textwidth]{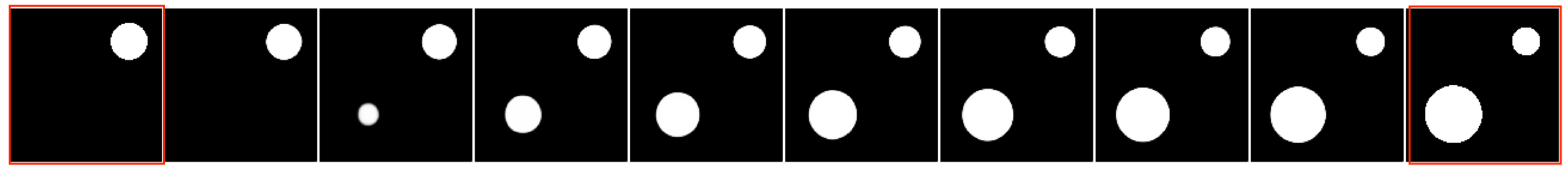}
        \caption{Interpolation results using a standard (vanilla) autoencoder. The changes between images are localized.}
        
        \label{fig:typical}
    \end{subfigure}

    \begin{subfigure}[b]{0.9\textwidth}
        \centering
        \includegraphics[width=0.9\textwidth]{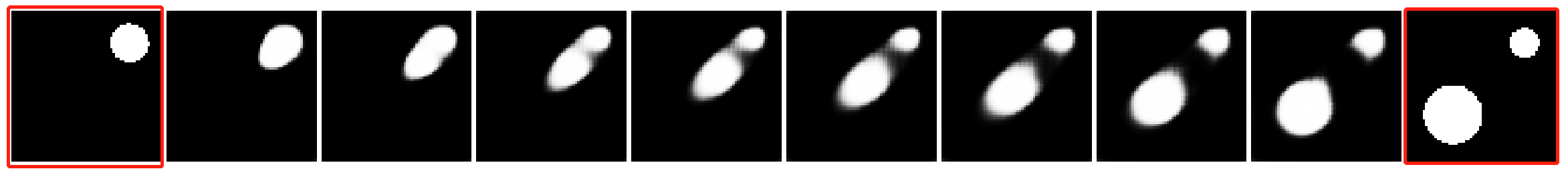}
        \caption{Interpolation results using our  method. The result shows a smooth transition between images that is consistent with the laws of physics.   }
        \label{fig:sub2}
    \end{subfigure}
\caption{Comparative analysis of image sequence interpolation. 
Given the images marked in red, the interpolated images are generated by decoding a linear combination of their latent codes after training.  }
    
    \label{fig:comp}
    
\end{figure*}

Inspired by this,
we define a path energy term by eliminating the momentum variable in the original dynamic OT model, and add it to the autoencoder as a regularization term to measure the performance of the generated path.
In this way, the interpolated path is trained to be consistent with physical principles.
%Our path energy term keeps the good side of DOT, and can encourage 
Our contribution is summarized as follows:
\begin{itemize}
    \item Our regularization term is mathematically explainable and encourages the autoencoder to generate geodesic paths between images, even in an environment with obstacles. 
    This also brings a continuous and smooth morphing along the interpolated images.
    \item  We study the properties of the path energy term. It has a closed form in a discrete setting. Gradient properties are given to facilitate training.
    \item   We evaluate the effectiveness of our approach in a variety of scenarios. It is noteworthy that traditional deep learning techniques usually require a considerable amount of training data to accurately learn features. Additionally, conventional optimization techniques are usually restricted to interpolating between individual pairs of images. On the contrary, our method produces robust and smooth interpolation results in all cases.
    \item Our method can be regarded as a variant of dynamic OT, where the path variable is parameterized by the weight parameters of the autoencoder network. As we know, this is the first successful attempt to apply optimal transport with a large deep learning training set on interpolation problem.
\end{itemize}

%produce reliable and smooth result  enables faithful interpolation between data

Generative AI models, such as Stable Diffusion, DALL-E and ChatGPT, have been the focus of much attention in recent years. Despite their impressive results, these models require a large amount of training data. Unfortunately, in many cases, obtaining such data can be difficult, costly, or even impossible due to privacy issues. Our method is a type of physics-driven deep learning\cite{cai2021physics}, which is becoming popular in the fields of machine learning and computational physics recently. It can produce more reliable and realistic results, particularly in situations where data is scarce.

In section 2, we will discuss some related research. Section 3 will present our proposed approach, and section 4 will provide numerical results.

\subsection{Notation}

We briefly introduce some important notation of our work.
A standard autoencoder consists of an encoder
$z = \E (x)$ where $z$ is the latent code
and a decoder $\hat{x} = \D(z)$.
In this context, $z$ represents the {\em latent code}, a lower-dimensional representation of the input data $x$. 
The encoder and decoder are trained simultaneously to recover the input data,
most commonly by minimizing a loss function $||x-\hat{x}||^2$.

In this paper, 
we focus on image interpolation using an autoencoder by decoding a linear interpolation of two latent codes.
Specifically,
for a pair of images $x_i, x_j$ with the corresponding latent codes $z_i$ and $z_j$,
we define a new latent code 
$$
z_{\itoj}(t) = (1-t) z_i + t z_j,
\quad
t \in [0,1].
$$
Then, the output
$ \D( z_{\itoj}(t)  )$ by decoding $z_{\itoj}(t)$
describes a path showing how pixels from $x_i$ move to $x_j$ when $t$ increases from 0 to 1.

\section{Related Work}

\subsection{Classic Optimal Transport and Wasserstein Distance}

Optimal transport (OT) is a well-developed mathematical theory that defines a family of metrics between probability distributions,
with a wide range of potential applications and extensions~\cite{villani2009optimal,peyre2019computational}.
In this work, we only consider the problem of two-dimensional image interpolation.
Thus, without loss of generality, we
define two density functions $\rho_0(s) \geq 0$ and $\rho_T(s) \geq 0$ with $s \in [0,1]^2$, which we assume to be bounded by total mass one
$$
\int_{[0,1]^2} \rho_0(s) d s=\int_{[0,1]^2} \rho_T(s) d s=1.
$$
The associated
$L^p(p \geq 1)$  Wasserstein (or Kantorovich-Rubinstein) distance between $\rho_0$ and $\rho_T$ is defined by:
$$
d_p\left(\rho_0, \rho_T\right)^p=\inf _M \int|M(s)-s|^p \rho_0(s) d s,
$$
where the infimum is taken over all maps $M : [0,1]^2 \to [0,1]^2$ that push forward the measure $\rho_0 ds$ to $\rho_T ds$. 
%When $p=2$, the distance defined above is often called the earth mover distance and is directly related to dynamic optimal transport.
The map $M$ that minimizes this objective, denoted as $M^*$, is known as the optimal transport map.
To interpolate between images, we can compute the optimal transport map and push forward the source image to the target image according to this map~\cite{bonneel2011displacement}:
$$
\rho(t, s)=\rho_0\left(M_t(s)\right)\left|\operatorname{det}\left(\partial M_t(s)\right)\right|, 
$$
where $M_t=(1-t) \operatorname{I}+t M^*$.
Here, $\rho(t, s)$ is the geodesic path between $\rho_0$ and $\rho_T$. In the context of image interpolation, $\rho(t)$ represents the interpolated image at any given time $t$.

\subsection{Dymamic Optimal Transport}

%We fix a time interval $[0, T]$ and consider all possible smooth enough, time-dependent, density, and velocity fields, $\rho(t, s) \geq 0, \mathbf{v}(t, s)$. 
Dynamic optimal transport, initially proposed in \cite{benamou2000}, resets the mass transfer problem into a continuum mechanics framework.
The path $\rho(t)$ is described by advecting the measure using a vector field $\mathbf{v}(t)$,
with $\mathbf{v}$ and $\rho$ satisfying the continuity equation, AKA, the conservation of mass formula.
Dynamic optimal transport finds the geodesic path $\rho(t, s)$ between $\rho_0(s)$ and $\rho_T(s)$ by minimizing the kinetic energy along the path:
\begin{equation}
\begin{aligned}
\min_{\rho, \mathbf{v}} \quad & \frac{1}{2} \int_0^T  \int_{ [0,1] ^2} \rho(t, s)|\mathbf{v}(t, s)|^2 d s d t,\\
\textrm{s.t.} \quad & \partial_t \rho+\nabla \cdot(\rho \mathbf{v})=0, ,\\
  &\rho(0, \cdot)=\rho_0, \quad  \rho(T, \cdot)=\rho_T,    \\
\end{aligned}
\label{equ:original}
\end{equation}
where the first constraint is  the continuity equation, and the remaining are the initial and final conditions
Proper boundary conditions in the velocity field should be considered.
In our work, we consider the Dirichlet boundary condition, specifically zero at the boundary positions.
Extension to more complex fluid mechanics models and combining it with other metrics such as $L^2$ distance~\cite{benamou2000,papadakis2014optimal} is flexible,
by changing the objective function or constraints.

After introducing a new variable, momentum $\m=\rho \mathbf{v}$, 
the above problem can be reformulated as a convex problem as follows:
\begin{equation}
\begin{aligned}
\min_{\rho, \m} \quad & \frac{1}{2} \int_0^T \int_{ [0,1]^2}  \frac{|\m(t, s)|^2}{\rho(t, s) }d s d t,\\
\textrm{s.t.} \quad & \partial_t \rho+\nabla \cdot(\m)=0, \\
  &\rho(0, \cdot)=\rho_0, \quad \rho(T, \cdot)=\rho_T.    \\
\end{aligned}
\label{equ:convexreformulation}
\end{equation}

This problem remains challenging to solve since the objective function is nonsmooth.
After discretization, 
there are several ways to solve it,
such as the Douglas-Rachford algorithm and the
primal-dual method in \cite{papadakis2014optimal},
as well as the alternating direction multiplier method (ADMM) in \cite{benamou2000}.
The drawback is that it is computationally expensive.

The square of the $L^2$-Wasserstein distance is equal to $2T$ times the infimum of dynamic optimal transport defined in \eqref{equ:original}~\cite{benamou2000}.
Dynamic OT provides the geodesic point of view of OT
and allows us to explore the path space.
In the context of image sequence interpolation,
the optimal solution $\rho(t)$ gives the interpolated sequence directly. 
%We refer the interested reader to \cite{villani2009optimal,papadakis2014optimal,peyre2019computational} for more details.

\subsection{Adversarial Regularizer}
The adversarial regularizer-based method is an important branch of autoencoder image interpolation. This technique utilizes generative adversarial networks (GANs) \cite{goodfellow2020generative} to produce high-quality interpolated images.
The generative model, here an autoencoder, produces the interpolated images, while the discriminative model, referred to as the critic, evaluates the generated images with certain criteria such as the similarity between the interpolated image and the
training data.
The output of the critic network is added as a regularization term for the autoencoder loss function,
helping it to refine the interpolated images further through joint training. 
Key research in this area includes works by \cite{advae,larsen2016autoencoding,
 berthelot*2018understanding, beckham2019adversarial}.
Despite the promising results these methods have achieved in refining image interpolation techniques, their data-driven approach significantly differs from our method which leverages robust mathematical models to ease the need for large amount of training data.

\section{Path Energy}

In this section, we explain the motivation to define path energy and how to apply it to the autoencoder, encouraging the interpolated results to be realistic.

As discussed earlier, given two images $x_i$ and $x_j$, the
autoencoder has the ability to generate a path $\D( z_{\itoj}(t)  ), t \in [0,1]$ by decoding a linear combination of their latent codes.
However, 
these generated images are usually not as realistic as expected. See Figure \ref{fig:typical} for an example.
Thus, we need an evaluation metric to penalize some "bad" paths.

The image interpolation problem can be interpreted as a mass transfer problem. 
For a gray image,
the value of the pixel represents the mass at that position.
Thus, dynamic OT can be adapted to the image interpolation problem.
The problem \eqref{equ:convexreformulation} minimizes over two variables $\rho, \m$,
with $\rho$ representing a path connecting $x_i$ and $x_j$ and $\mathbf{m}$ describing the momentum of fluid movement.
A natural idea is that for any given $\rho$, we can optimize through $\m$ to find the least motion momentum and use the kinetic energy as the path energy value associated with $\rho$.

Formally, for any given $\rho(t,s)$ satisfying $\rho(0, \cdot)=\rho_0, \rho(T, \cdot)=\rho_T$, we can define the path energy  over $\rho(t,s)$ based on normal optimal transport as  
\begin{equation}
\begin{aligned}
J(\rho) =\min_{\m} \quad &  \frac{1}{2}  \int_0^T \int_{ [0,1] ^2} \frac{|\m(t, s)|^2}{\rho(t, s) }d s d t\\
\textrm{s.t.} \quad & \partial_t \rho+\nabla \cdot(\m)=0. \\
\end{aligned}
\label{equ:pathenergy}
\end{equation}
Note that $J(\rho)$ is convex as \eqref{equ:convexreformulation} is a bivariate convex problem.

We consider a variation in this paper in which there are obstacles in the environment. \cite{papadakis2014optimal} proposed a generalized cost function with spatially varying weights, which can be accomplished by adding a constraint to make the momentum zero at all obstacle positions: $$ m(t,s) = 0,\quad \forall s \in C, $$ where $C$ is the set of obstacle positions.

Another variation we consider in this paper is unbalanced OT to address the case when the source and target images have different masses.
A general
methodology consists of introducing a source term $\mathfrak{s}$ into the continuity equation and penalizing it in the same way as the momentum \cite{chizat2018interpolating}. 
Formally,  the path energy  over $\rho(t,s)$ based on unbalanced OT is 
\begin{equation}
\begin{aligned}
J(\rho) =\min_{\m} \quad &  \frac{1}{2}  \int_0^T \int_{ [0,1] ^2} \frac{|\m(t, s)|^2}{\rho(t, s) } + \tau  \frac{|\mathfrak{s}(t, s)|^2}{\rho(t, s) }  d s d t\\
\textrm{s.t.} \quad & \partial_t \rho+\nabla \cdot(\m)=\mathfrak{s}, \\
\end{aligned}
\end{equation}
where $\tau$ is the source term weight.

\subsection{ Discretization and Solver}

Here, we discretize the time to $\{0,1,2,.., T\}$,
and assume that the space is uniformly discreted at the points $(i/n,j/n)\in [0,1]^2$.
Source and target densities are represented as $\rho_0, \rho_t \in \R^{n,n}$ and a path $\rho \in \R^{T+1,n,n}$.
We adopt a staggered grid discretization scheme, which is commonly used in fluid dynamics~\cite{papadakis2014optimal}.
The momentum vector $\m$ and its corresponding weight vector $\Q$ will be defined at half-grid points in each direction of space at time $t$.
We denote $\m_t = (\m^1_t, \m^2_t)$ where $\m^1_t\in \mathbb{R}^{n+1,n}$ and $\m^2_t  \in \mathbb{R}^{n,n+1}$
and $\Q_t = (\Q^1_t,\Q^2_t)$ where $\Q^1_t\in \mathbb{R}^{n+1,n}$ and $\Q^2_t  \in \mathbb{R}^{n,n+1}$.
Given a path $\rho$, $\Q_t$ is deterministicaly defined as follows
$$
\Q^1_{t,i,j} = \frac{2}{\rho_{t,i,j}  + \rho_{t,i+1,j} },
\quad \quad
\Q^2_{t,i,j} =  \frac{2}{\rho_{t,i,j}  + \rho_{t,i,j+1} }.
$$
Using the staggered grid discretization scheme, divergence operator 
associated with a vector field $\m_t $ in the linear constraint of problem \eqref{equ:pathenergy} is defined as:
$$
(\nabla \cdot \m_t)_{i,j} = \m^1_{t,i+1,j} - \m^1_{t,i,j} + \m^2_{t,i,j+1} - \m^2_{t,i,j}.
$$

After flattening the vectors $\m$ and $\Q$,
problem \eqref{equ:pathenergy} becomes:
\begin{equation*}
\begin{aligned}
J(\rho) = \min_{\m} \quad &  \sum_{t=0,1,2,..T-1} \m_t^T \operatorname{Diag} (\Q_t) \m_t\\
\textrm{s.t.} \quad & \nabla \cdot(\m_t)=b_t, t=0, 1,2,...T-1 \\
\end{aligned}
%\label{equ:quadraticmomentum}
\end{equation*}
where $b_t=-\partial_t \rho = \rho_t - \rho_{t+1}$.

This is a quadratic problem with linear constraint, and its optimality condition (KKT condition) is given by 
\begin{equation}
\left[\begin{array}{cc}
\operatorname{Diag}(\Q_t) & \nabla \cdot ^{\top} \\
\nabla \cdot & 0
\end{array}\right]\left[\begin{array}{c}
\mathbf{m}_t \\
\lambda_t
\end{array}\right]=\left[\begin{array}{c}
\mathbf{0} \\
\mathbf{b_t}
\end{array}\right],
t=0,1,2,...,T-1,
    \label{equ:kkt}
\end{equation}
where $\lambda_t$ is the Lagrange multiplier.
After solving for $\m$ in the above linear system, we get a closed formula for the path energy defined in \eqref{equ:pathenergy}:
\begin{equation*}
\begin{aligned}
J(\rho) = &\sum_{t=0}^{T-1} b_t^T \left(\mathbf{\nabla \cdot \operatorname{Diag}(\Q_t)}^{-1} \mathbf{\nabla \cdot ^{\top} }\right)^{-1} b_t,
\end{aligned}
\label{equ:quadraticmomentum}
\end{equation*}
by assuming $\Q_t$ are positive.
The extensions to the obstacle case and unbalanced OT case are similar and can be found in \Cref{sec:uotandobs} in the supplementary material.

\subsection{Gradient Computation}

First, we denote $\nabla \cdot \operatorname{Diag}(\Q_t)^{-1} \nabla \cdot  ^{\top} $ as $A_t$,
which is a reweighted poison operator.
There are $T$ sparse linear systems to solve in $J(\rho)$:
\begin{equation}
A_t y =   b_t, 
\quad
t = 0, 1, 2, ... T-1.
\label{linearsystem}
\end{equation}
We show the gradient property of our path energy term in the following theorem,
which is important in the backpropagation of neural network training.

\iffalse
$$
(\frac{\partial J}{\partial \rho_t})_{i,j} = - \frac{ ( y_{t,i+1,j} -y_{t,i,j} )^2+(y_{t,i-1,j} - y_{t,i,j})^2+(y_{t,i,j+1}-y_{t,i,j})^2+(y_{t,i,j-1}-y_{t,i,j})^2 } {4} +
 2y_{t,i,j}, 
$$
\fi 
\begin{theorem}
The first-order gradient of the  path energy defined  in \eqref{equ:pathenergy} at time $t = 0, 1, 2, ... T-1$ is given by 
$$
(\frac{\partial J}{\partial \rho_t})_{i,j} = - \frac{1}{4} \sum_{(k,l) \in \mathcal{O}_{i,j}}  ( y_{t,k,l} -y_{t,i,j} )^2  +
 2y_{t,i,j}, 
$$
where $\mathcal{O}_{i,j}$ is the connected neighbor of $(i,j)$, i.e., $\{ (i-1,j),(i+1,j),(i,j-1),(i,j+1) \}$ and $y_t \in \R^{n\times n}$ is the solution of equation \eqref{linearsystem}.

\end{theorem}
\begin{proof}
At time t,
let $\mathbf{u}$ represent the elementwise inverse of $\w$, that is, each element of $\mathbf{u}$ is the reciprocal of the corresponding element in $\w$.
Then, we compute
$$
\begin{aligned}
\frac{\partial J}{\partial \mathbf{u}} &=\frac{\partial}{\partial \mathbf{u}} b_t^{\top} A^{-1} b_t \\
& =b_t^{\top} ( \frac{\partial}{\partial \mathbf{u}} A_t^{-1} ) b_t \\
 &=-(y_t^{\top} \mathbf{\nabla \cdot } )\frac{\partial}{\partial \mathbf{u}}  \operatorname{Diag}(\mathbf{u}) (\mathbf{\nabla \cdot ^{\top} }  y_t ) \\
 &=- (\nabla y_t) \odot (\nabla y_t),
\end{aligned}
$$
where $\odot$ is the Hadamard (entrywise) product.
Note that the transpose of a discrete divergence operator is the gradient operator.

Secondly, $\frac{\partial J_t}{\partial b_t} = (A_t^{-1} + (A_t^{-1})^{T}) b_t = 2 A_t^{-1}b_t = 2y_t$.
Using the chain rule, the result is straightforward.
\end{proof}
%The extension to a general cost function is straightforward.

\iffalse
not very clear on the continuous space property
\subsection{Path Energy in the Continuous Space}

we have similar results in continuous space.
The Lagrange function of \eqref{equ:pathenergy} is 
$$
L(\m, \lambda) = \int_0^T \int_{ [0,1] ^2} \frac{|\m(t, s)|^2}{\rho(t, s) }d s d t
+ <\lambda,\partial_t \rho+\nabla \cdot(\m)>.
$$
The first-order optimality condition is given by 
$$
\frac{\partial L}{\partial m}
=\frac{\m}{\rho} - \nabla \lambda = 0,
\quad
\frac{\partial L}{\partial \lambda} =\partial_t \rho+\nabla \cdot(\m)=0.
$$
Let's define a linear operator $\mathcal{A}_t$ depending on $\rho(t)$ as:
$$
\mathcal{A}_t y = \nabla \cdot \frac{( \nabla \cdot )^T y  }{\rho(t) },
\forall y
$$
Assuming that $y(t)$ is the solution of $\mathcal{A}_t y= -\partial \rho(t)$.
Then we have
$$
J(\rho) = \int_0^T \int_{ [0,1] ^2} \rho(t, s) (\nabla y(t,s))^2 d s d t
$$
This result is important when we analyze the behavior of the path energy without loss of information.

NEED TO CHECK THE DIVGENCE AND GRADIENT OPERATOR
\fi

%Not clear
\subsection{Proposed Algorithm}
\label{algosection}
The goal of this work is to improve the autoencoder image interpolation result by $\D( z_{\itoj}(t)  )$, which decodes a linear combination of latent codes  of image $x_i$ and $x_j$. 
To that end,
we propose a new approach, which adds the path energy defined in \eqref{equ:pathenergy} as a regularization term when reconstructing the input using an autoencoder.
This encourages $\D( z_{\itoj}(t)  )$  to reach the geodesic path between $x_i$ and $x_j$,
and also brings about a smooth transportation effect, which will be shown in the experiments.
Here $x_i$ and $x_j$ can be any pair of data in the training dataset.

However,
there are some issues with the path $\D( z_{\itoj}(t)  )$ generated by the autoencoder.
First,
the generated path does not necessarily meet the mass-preserving properties, which are required when adapting the image interpolation problem to a mass transfer problem in the normal dynamic OT setting.
The initial dimensions of the pictures may be too big to be plugged into the path energy term.
To address this, downsampling the path $\D( z_{\itoj}(t)  )$ would be beneficial. 
Furthermore, to ensure that the KKT system \eqref{equ:kkt} has a solution,
the weight $\mathbf{w}_t$ must be positive.
Thus, some pre-processing on $\D( z_{\itoj}(t)  )$ is necessary.

The full procedure of our algorithm is presented in Algorithm \ref{alg}. The Sigmoid activation at the last layer of the autoencoder guarantees a non-negative output. Therefore, we measure the reconstruction loss by computing the binary cross entropy BCE$(x_i,\hat{x}_i)$ between the input data $x_i$ and the autoencoder output $\hat{x}_i$.
$J$ is the path energy defined in \eqref{equ:pathenergy},  $D_{\itoj}$ refers to the path $\D( z_{\itoj}(t)  ), t\in [0,1]$ after preprocessing, and 
\small
     \[
     \begin{aligned}
         \MassLoss_{\itoj} &= \sum_{t=1}^{T-1} | \Mass( \D( z_{\itoj}(t)  ) )\\
         &- (t \Mass(x_i)+(1-t) \Mass(x_j)) |,
     \end{aligned}         
     \]
\normalsize
with $
\Mass (x)= \sum_{i,j} |x_{i,j}|. 
$
The term $\MassLoss_{\itoj}$ is used to constrain the mass of the generated path.
For computational efficiency,
the choice of the indices $i,j$ can be a random subset of the entire data set.
For motion-dominated interpolation problems, we use the path energy based on normal optimal transport, and for shape-dominated interpolation problems, we use the path energy based on unbalanced optimal transport (see  \Cref{otvsuot} in supplementary materials for more explanation).

After training,
we generate interpolation images between $x_i$ and $x_j$ by decoding a convex combination of their latent codes.

\begin{algorithm}
\caption{Our proposed method for autoencoder image interpolation}\label{alg}
\begin{algorithmic}[1]
\State \textbf{Input:} training image dataset X, an autoencoder with Sigmoid activation at the last layer
\State \textbf{Step 1: training process:}
\For{each epoch}
    \State generate $D( z_{i \to j}(t) ), t=1,2,\dots,T-1$ for the chosen $i,j$ pairs
    \State (optional) downsample the $D( z_{i \to j}(t) )$
    \State threshold the $D( z_{i \to j}(t) )$ with positive value $\epsilon$
    \State (optional) normalize $D( z_{i \to j}(t) )$ with a standard mass
    \State compute the loss function as follows and backpropagate:
    \begin{equation}
    \scriptsize
        \sum_i \text{BCE}(x_i,\hat{x}_i) + \alpha \sum_{i,j} J( D_{i \to j} )  + \beta \sum_{i,j} \text{MassLoss}_{i \to j}
    \end{equation}
\EndFor
\State \textbf{Step 2: interpolation process:}
\State generate interpolated images between $x_i$ and $x_j$ by decoding a convex combination of their latent codes 
\[ 
D( (1-t) z_i + t z_j ), \quad
0 < t < 1.
\]
\end{algorithmic}
\end{algorithm}

\textbf{Note:} For more detailed explanations of the steps and the used functions, please refer to Section \ref{algosection}.

\section{Numerical Experiments}
In this section, we evaluate the effectiveness of our approach in a variety of scenarios, from the most extreme (two training samples) to the moderate (a few training samples) to the large (MNIST). It is noteworthy that traditional deep learning techniques usually require a considerable amount of training data to accurately learn features. Additionally, conventional optimization techniques are usually restricted to interpolating between individual pairs of images. In contrast, our method, based on the principles of physics-driven deep learning, produces robust and smooth interpolation results in all cases.

%\subsection{Network Setting and Hyperparameter}
The experimental setup is as follows, unless otherwise specified.
For all experiments,
we adopt an autoencoder architecture similar to that of ACAI~\cite{berthelot*2018understanding}.
The encoder consists of 3 blocks of two $3 \times 3$ convolutional layers followed by a $2 \times 2$ average pooling. 
The number of channels is doubled before each average pooling layer.
This is followed by two more $3 \times 3$ convolutional layers,
and the final output is used as the latent code whose dimension is $16\times 4 \times 4$.
All of the convolution layers, except the final one, use a leaky ReLU activation.
The decoder has 3 similar blocks with average pooling replaced by upsampling layer and channel halved.
Two more convolutional layers are then performed, the last having a sigmoid activation.
We implemented our codes in Pytorch and use the package scipy.sparse.spsolve to solve linear systems in the path energy term. 
The parameters are optimized with Adam with the default parameters.

The threshold value $\epsilon$ is set to $1e-5$.
We adjust the values of the path energy weight $\alpha$, mass weight $\beta$, and source weight $\tau$ to obtain the best performance. All experiments use a downsampling size of 32 $\times$ 32, and the time discretization number $T$ is chosen between 10 to 30.
\Cref{Tchoice} in supplementary materials discussed how to choose the right $T$.

%The input data are binary images.
%All image sizes are 64 $\times$ 64,
%except that the MNIST dataset was resized to 32 $\times$ 32.
%The value $\epsilon$ to threshold the generated path is 1e-2 and the standard mass is set to 500.
%Furthermore, .

We quantify the quality of the interpolated images by using the Structural Similarity Index (SSIM) score in some experiments. This score evaluates the similarity between two images in terms of structure, brightness, and contrast, providing a more perceptually relevant assessment of the image quality. It is worth noting that if the SSIM score is equal to one, it means that the two images are identical.
Since we have a sequence of images,
we report the mean of the SSIM score of any two adjacent images.

\subsection{Ablation Study}
\label{ablation}
This ablation study aims to show the improvements brought by our regularization term, particularly when the training dataset is limited to only two data points. We compare our algorithm to a standard baseline autoencoder across different image types: gray-scale, binary, and RGB. We use normal optimal transport based path energy for the first two examples and unbalanced optimal transport based path energy for the last one. The results of these comparisons are presented in \Cref{fig:comparison}. 

Our method produces significantly better results than the baseline autoencoder, which typically yields local alterations, resulting in a cross-dissolving effect in the images. 
In the grayscale image, our technique effectively shifts the pixel density from one corner to the other. In the binary image scenario, our approach redistributes the mass from a central large circle into two distinct circles at opposite corners. For RGB images, our method achieves a smooth gradation in color transitions, as seen in the cartoon face. 
We also provided the SSIM score of the interpolated images as shown in the figure caption.
Even though our method does not always achieve a higher SSIM score, the visual improvement still demonstrates the effectiveness of our regularization term in generating realistic interpolated images that follow the laws of physics.

\begin{figure*}
\centering
\begin{tabular}{cc}
 \begin{subfigure}{\textwidth}
        \centering
        \includegraphics[width=0.4\linewidth]{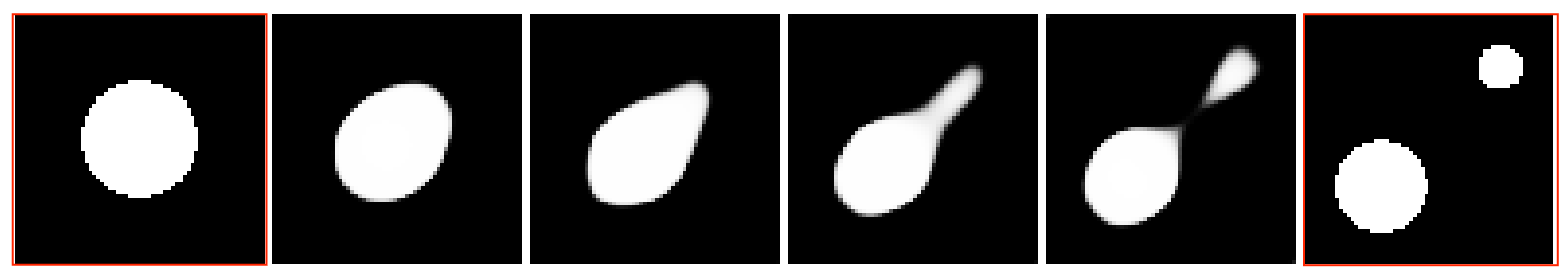}
        $\quad$
        \includegraphics[width=0.4\linewidth]{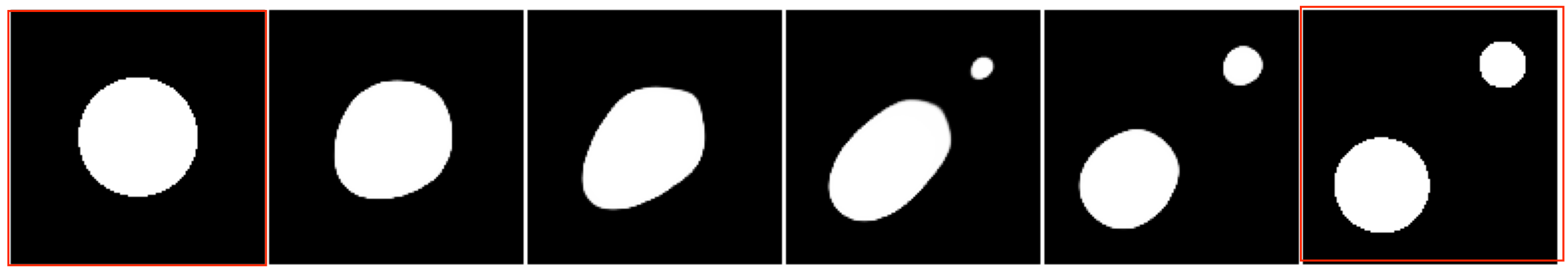}
        \caption{Interpolation results on binary images. Left: Our method (SSIM score: 0.87).
        Right: Baseline autoencoder (SSIM score: 0.91).}
        \label{fig:example2-2}
    \end{subfigure} \\
    \begin{subfigure}{\textwidth}
        \centering
        \includegraphics[width=0.4 \linewidth]{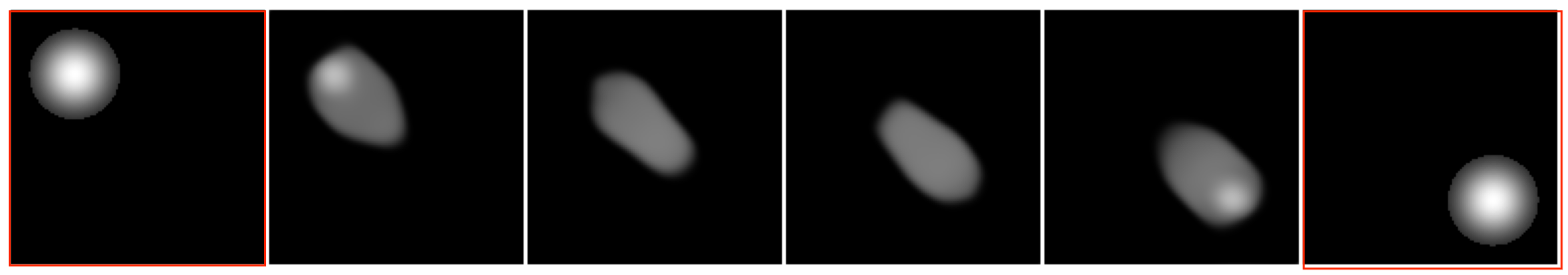}
        $\quad$
         \includegraphics[width=0.4 \linewidth]{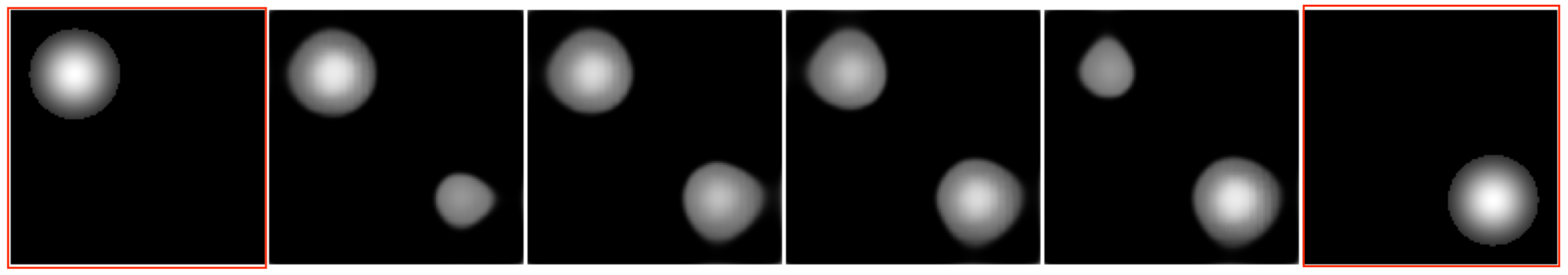}
        \caption{Interpolation results on gray-scale images. Left: Our method (SSIM score: 0.89).
        Right: Baseline autoencoder (SSIM score: 0.93).}
        \label{fig:example1-2}
    \end{subfigure} \\
   
    \begin{subfigure}{\textwidth}
        \centering
        \includegraphics[width=0.4\linewidth]{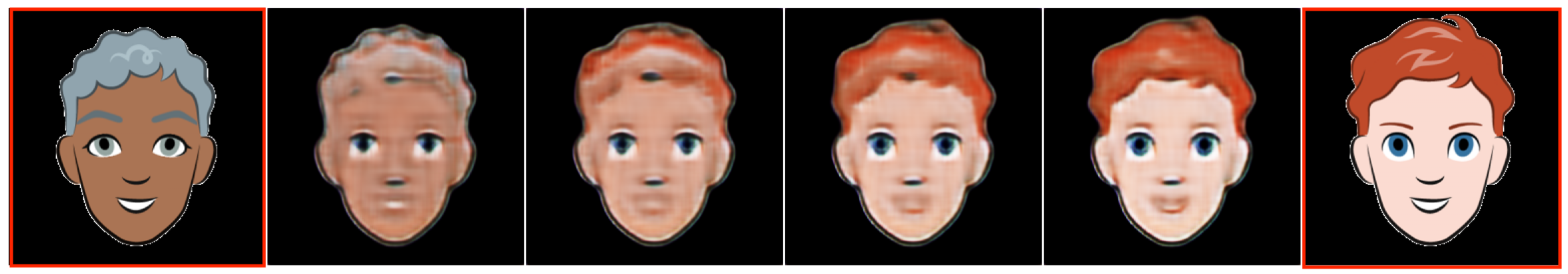}
        $\quad$
        \includegraphics[width=0.4\linewidth]{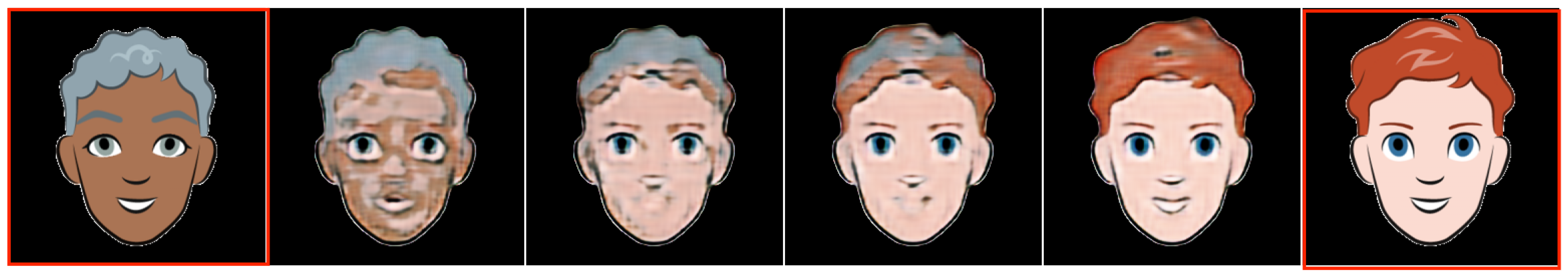}
        \caption{Interpolation results on RGB images. Left: Our method (SSIM score: 0.91).
        Right: Baseline autoencoder (SSIM score: 0.88).}
        \label{fig:example3-2}
    \end{subfigure}
\end{tabular}
\caption{Comparison of our proposed method and the baseline autoencoder method across different image types.}
\label{fig:comparison}
\end{figure*}

\subsection{Interpolation with Obstacle in the Environment }
In this example, we introduce obstacles into the environment. Figure \ref{fig:labyrinth} displays a labyrinth map with a circle moving from the bottom left corner to the right upper corner. The walls of the labyrinth are colored pink and the mass cannot pass across the wall. Subfigure (a) shows the result of our interpolation, which is very smooth, and the circle is able to squeeze through the narrow paths. To compare, we also conducted an experiment using the proximal splitting method to solve the original dynamic OT~\cite{papadakis2014optimal}. The result of this experiment is shown in subfigure (b). The mass transfer follows the geodesic path; however, the mass splits and goes to different paths, resulting in a visually messy interpolation. Additionally, the pixels carrying mass in the interpolated result have values 15 times higher than the original value, leading to a smaller object region. Therefore, the result of our method is much smoother than the one from the optimization method.

\begin{figure*}
\centering
\rotatebox[origin=c]{0}{(a)}\quad
{\includegraphics[width=0.8\textwidth,valign=c]{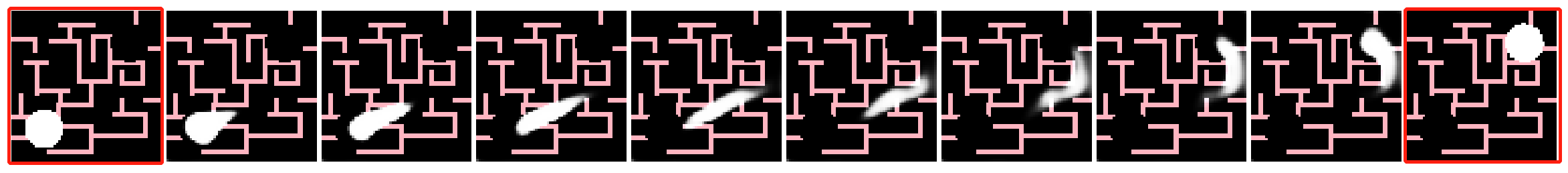}}

\rotatebox[origin=c]{0}{(b)}\quad
{\includegraphics[width=0.8\textwidth,valign=c]{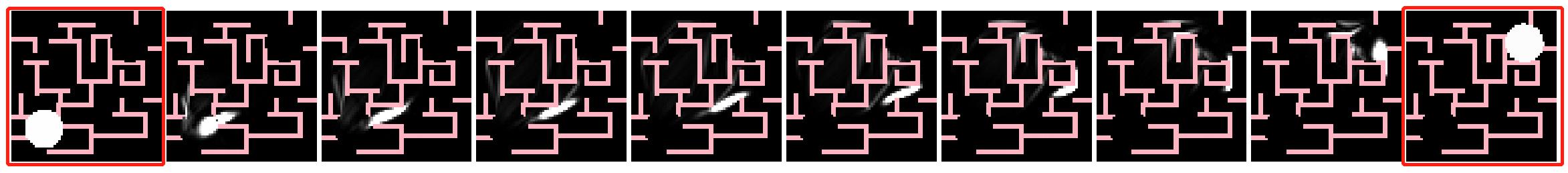}}

\caption{An example of image interpolation when obstacles in the environment (marked pink) are present. (a) The result of our proposed method. (b) The result of the proximal splitting method with a $64 \times 64 \times 120$ space-time discretization grid \cite{papadakis2014optimal}.}
\label{fig:labyrinth}
\end{figure*}

\subsection{Barycenter Problem }

In this section, we evaluate our approach in situations where data is limited, particularly focusing on the barycenter problem, which has recently become a major topic in optimal transport theory. A barycenter, or Wasserstein barycenter, in this context is a distribution that minimizes the weighted sum of Wasserstein distances to a set of given distributions. In image processing, this concept allows us to treat each image as a distribution. Traditional optimization techniques calculate the transportation plan to obtain the barycenter, a process that is computationally intensive. However, using an autoencoder, the barycenter can be efficiently obtained by decoding a convex combination of the latent codes after training. For example, with just four training images represented by latent codes \(z_1, z_2, z_3, z_4\), the barycenter at the \(i\)-th row and \(j\)-th column of a \(6 \times 6\) grid can be generated by decoding:
\small
$$ 
\D \left( (1-\frac{i}{6})(1-\frac{j}{6})z_1 + \frac{i}{6}(1-\frac{j}{6})z_2 + (1-\frac{i}{6})\frac{j}{6} z_3 + \frac{i}{6}\frac{j}{6} z_4 \right). 
$$
\normalsize % Reset to normal size
The core is to ensure the reliability of the trained autoencoder in producing meaningful results.

We tested our method on two different barycenter problems, each using four training data points that represented different distributions. The first problem was related to shape analysis and involved four distinct shape images. 
Thus we used a path energy term based on unbalanced optimal transport.
The results, as seen in \Cref{fig:bary1}, showed sharply defined yet smooth barycenters. The second problem focused on position changes, using images of a circle at various locations. 
We used a path energy term based on normal optimal transport.
As shown in \Cref{fig:bary2}, our method was able to effectively maintain the circle's shape while generating it at different positions. Further details on the results obtained with other methods can be found in Appendix 7.4. To conclude, our approach to the barycenter problem successfully combines sharpness with smooth transitions, thus demonstrating its effectiveness.

\begin{figure}
\centering
\includegraphics[width=0.7\linewidth]{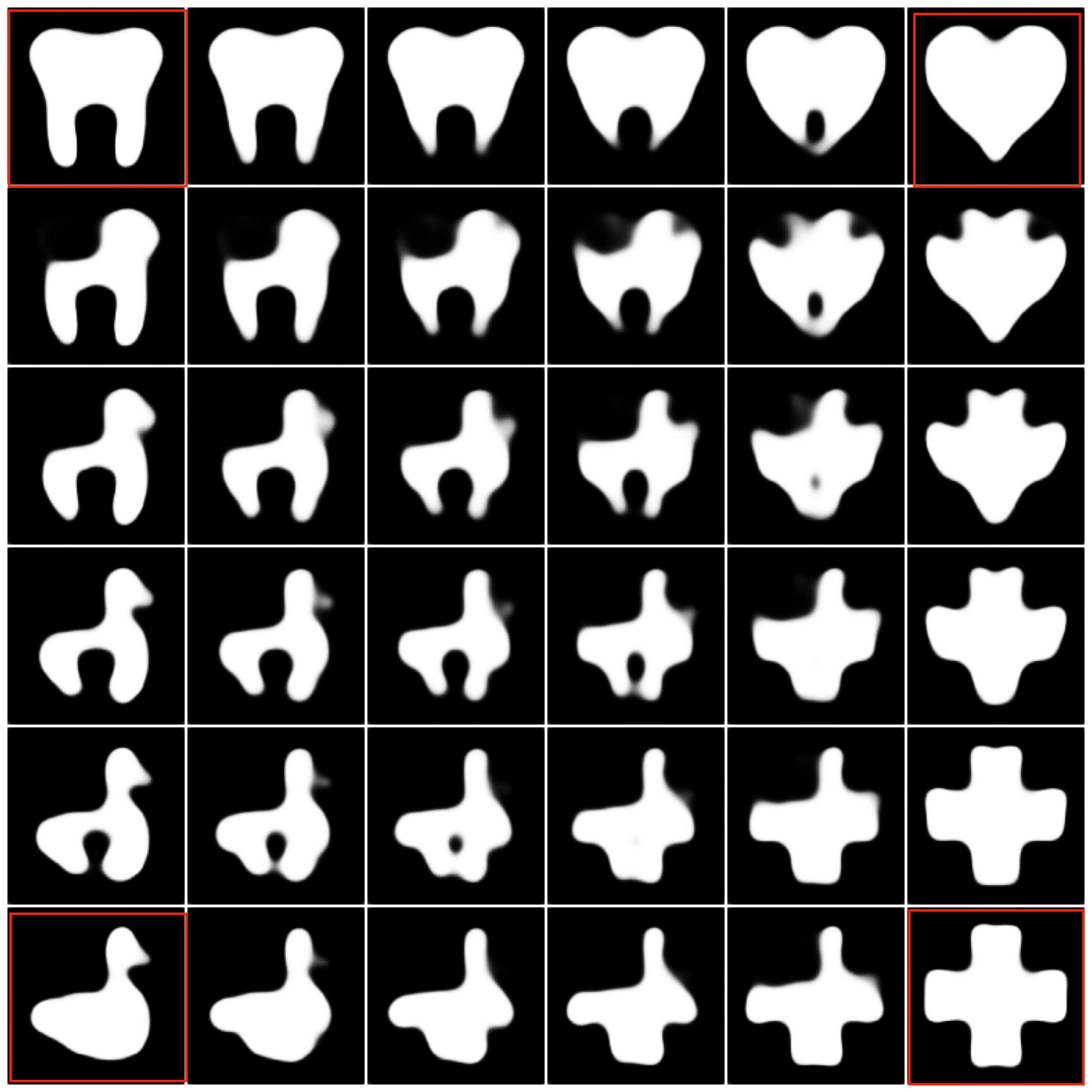}
\caption{\label{fig:bary1}
Barycenter example 1. 
There are only four images (marked red above) in the training dataset. Each remaining image was created by decoding a convex combination of their corresponding latent codes.}
\end{figure}

\begin{figure}
\centering
\includegraphics[width=0.7\linewidth]{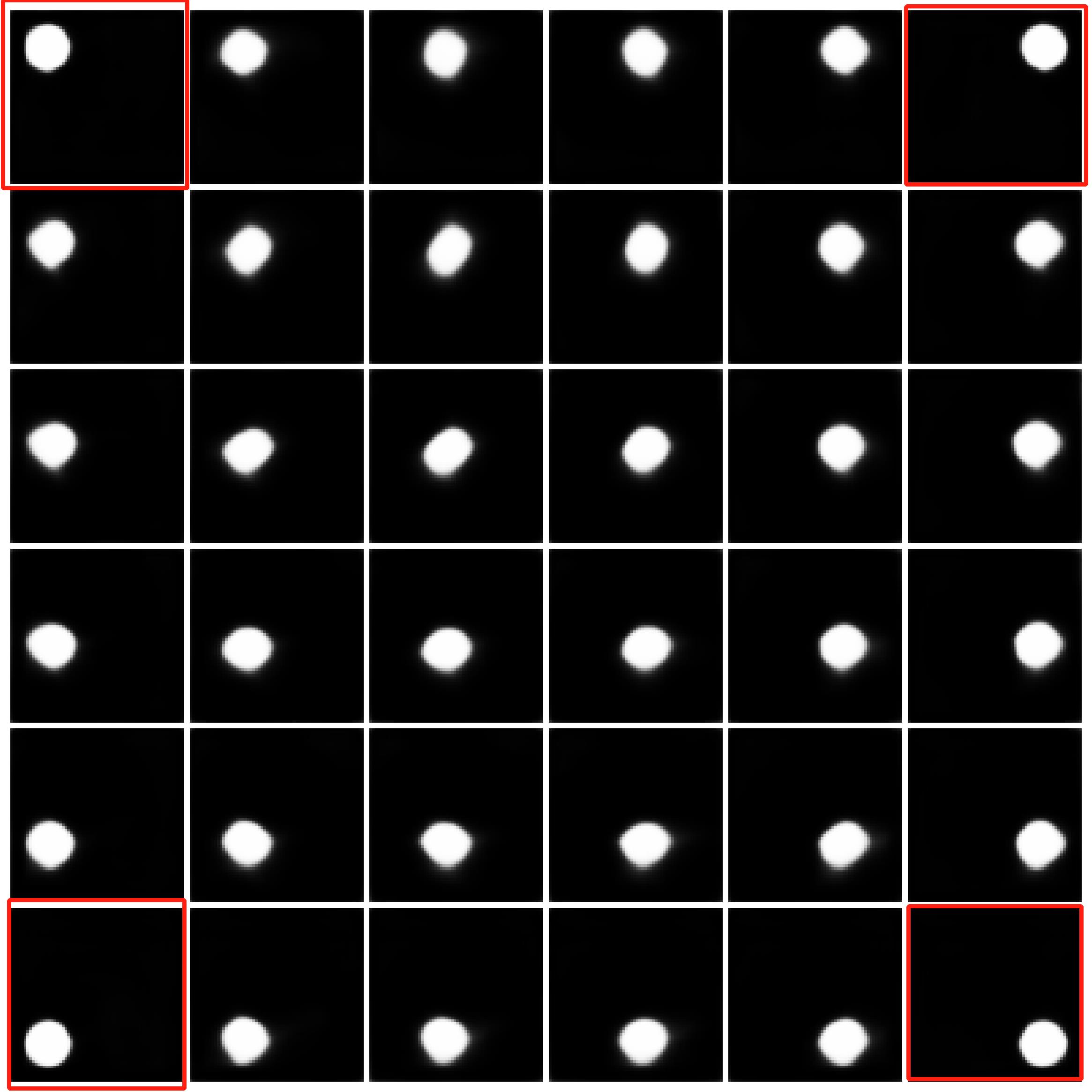}
\caption{\label{fig:bary2}
Barycenter example 2. 
There are only four images (marked red above) in the training dataset. Each remaining image was created by decoding a convex combination of their corresponding latent codes.}
\end{figure}
%make smaller if needed

\subsection{MNIST Dataset}
\begin{figure*}[!htbp]
    \centering
    \begin{subfigure}[b]{0.8\textwidth}
        \centering
        \includegraphics[width=\textwidth]{{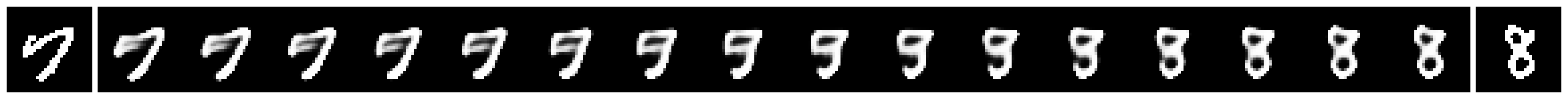}}
        \caption{ Ours (SSIM score mean: 0.96, std: 0.01).}        
        \label{fig:sub1}
    \end{subfigure}
    \begin{subfigure}[b]{0.8\textwidth}
        \centering
        \includegraphics[width=\textwidth]{{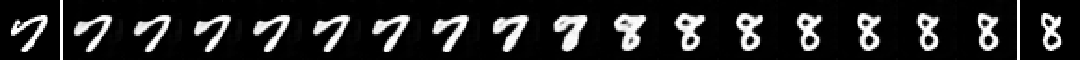}}
        \caption{ ACAI (SSIM score mean: 0.87, std: 0.10). Note that the transition is abrupt in the middle.}        
        \label{fig:sub2}
    \end{subfigure}
    \begin{subfigure}[b]{0.8\textwidth}
        \centering
        \includegraphics[width=\textwidth]{{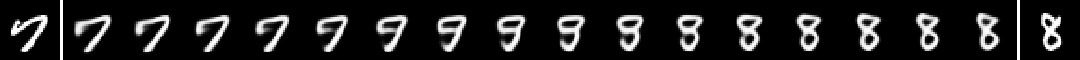}}
        \caption{ VAE (SSIM score mean: 0.9, std 0.10).}        
        \label{fig:sub3}
    \end{subfigure}
    \begin{subfigure}[b]{0.8\textwidth}
        \centering
        \includegraphics[width=\textwidth]{{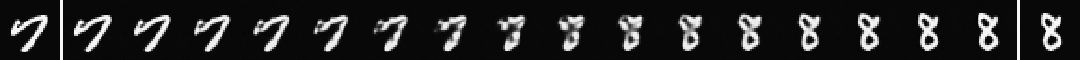}}
        \caption{ Baseline Autoencoder(SSIM score mean: 0.9, std: 0.11). Note that the interpolation results are  blurry. }        
        \label{fig:sub1}
    \end{subfigure}
\caption{Comparison of interpolation results on the MNIST dataset using four different autoencoder methods: Ours, ACAI, VAE, and a baseline autoencoder.
 }   
    \label{fig:mnistcomp}    
\end{figure*}
The purpose of this experiment is to show the interpolation performance of our method for a large real-world dataset benchmark.
The images were padded to size 32 $\times$ 32 for training.
We employ a batch size of 256 and randomly select 9 pairs of paths to penalize their energy term in the loss function instead of all pairs of paths to reduce computational cost.
We used a path energy term based on unbalanced optimal transport.
The results are shown in Figure \ref{fig:MNIST}.
We can see that the morphing between these handwritten digits
follows an optimal transportation path.

We compare our method to other state-of-the-art autoencoders, such as the Variational Autoencoder (VAE) \cite{kingma2014stochastic}, Adversarily Constrained Autoencoder Interpolation (ACAI) \cite{berthelot*2018understanding}, and a baseline autoencoder, as seen in Figure \ref{fig:mnistcomp}. The baseline model is a basic autoencoder with a BCE loss function. Our method and VAE show the most satisfactory visual interpolation results. 
Notably,
our method has a higher SSIM score and a much lower standard deviation. The interpolation using the ACAI method has abrupt transitions, likely due to its generated images being trained to follow the training data. In contrast, the baseline autoencoder usually produces blurry interpolations. Therefore, our method has a significant improvement in the interpolation task among autoencoder-based methods.

We also investigate how the latent space of an autoencoder is impacted by the regularization term. We analyze three distinct methods: the Kullback-Leibler (KL) divergence in VAE, the adversarial term in ACAI, and the path energy term in our proposed method. We draw inspiration from \cite{berthelot*2018understanding} and conduct the same classification and clustering experiments on the MNIST dataset using the latent spaces generated by these different autoencoders after training. The results of the experiments are presented in Tables \ref{tab:classification} and \ref{tab:clustering}, which demonstrate that ACAI outperforms the other methods in both classification and clustering tasks. Our method exhibits moderate performance, surpassing the Baseline. We believe that this result is reasonable. Unlike the KL divergence in VAEs, which enforces prior assumptions on the latent space, our method's path energy term works directly on the pixel space of reconstructed images, rather than on the latent codes. Consequently, while our method guarantees smooth interpolation and follows the principle of least path energy, it does not necessarily disentangle essential representations within the dataset, such as class identity.

\begin{figure}
\centering
\includegraphics[width=0.8\linewidth]{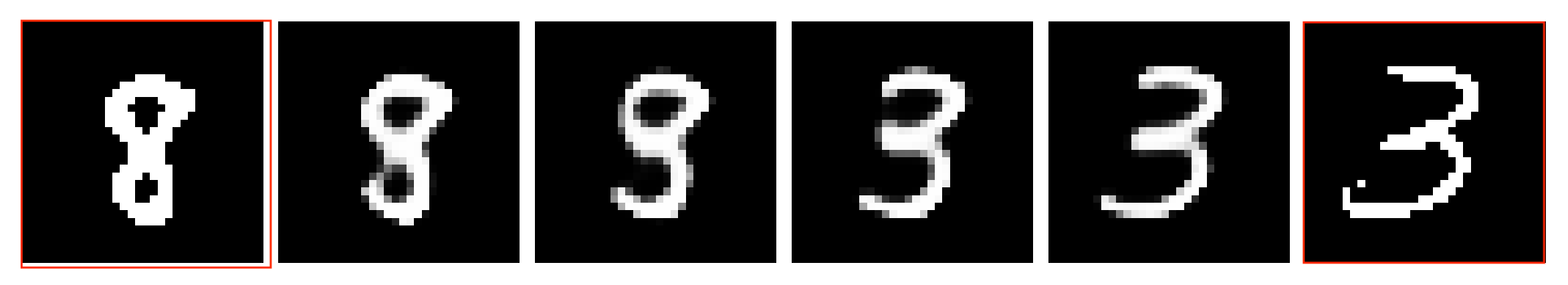}
\includegraphics[width=0.8\linewidth]{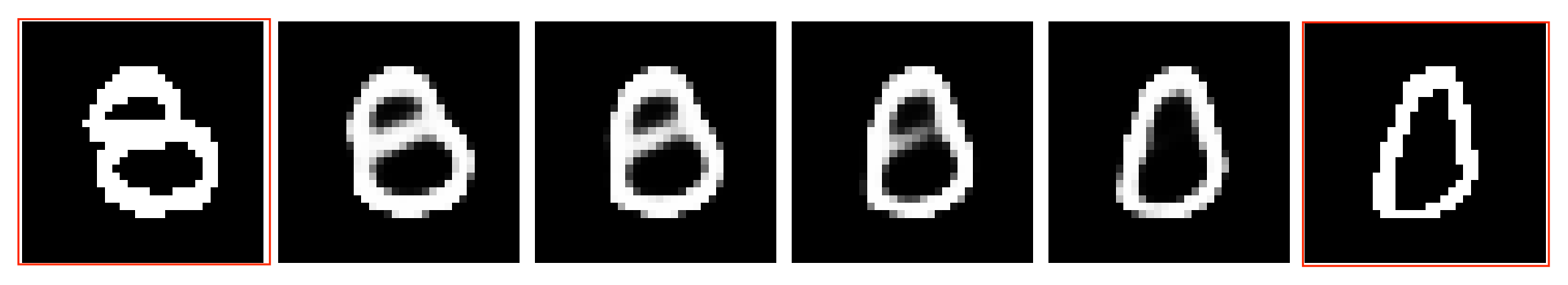}
\includegraphics[width=0.8\linewidth]{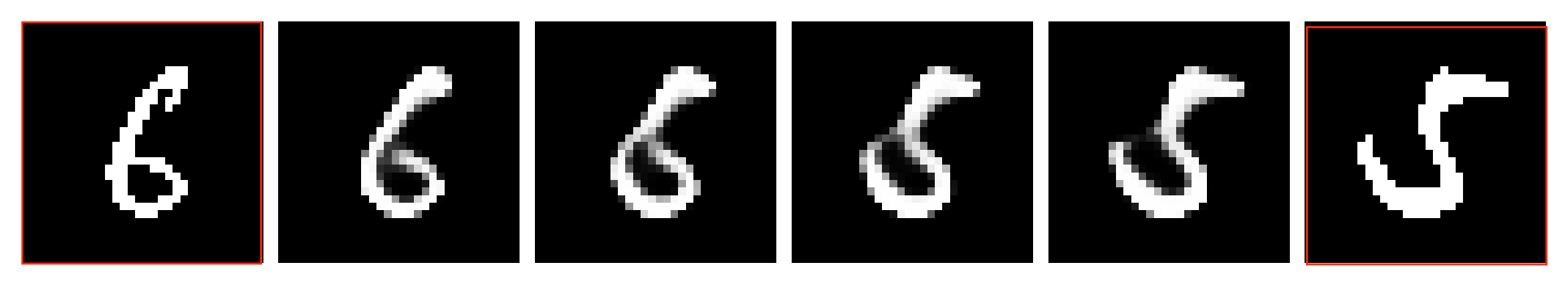}
\includegraphics[width=0.8\linewidth]{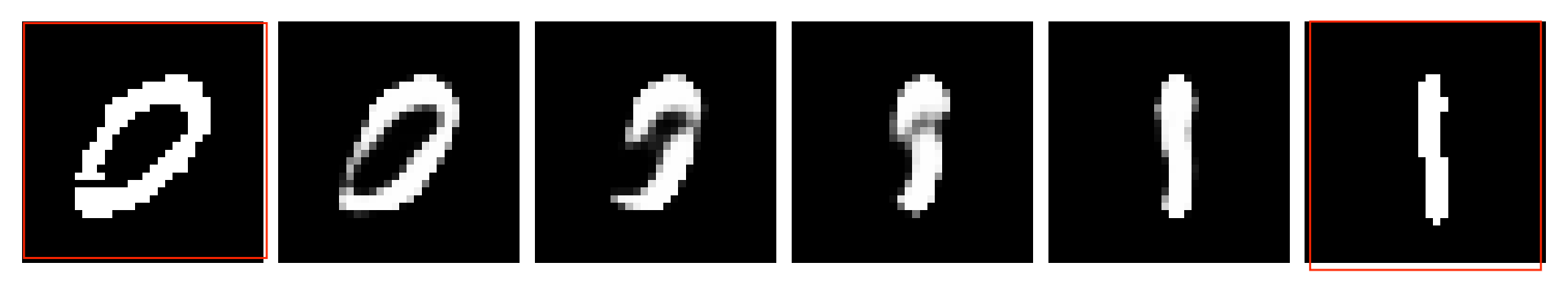}
\caption{\label{fig:MNIST}The interpolation results on MNIST dataset using our proposed method.}
\end{figure}

\begin{table}
  \centering
  \begin{tabular}{@{}lcccc@{}}
    \toprule
    Baseline  & VAE & ACAI & ours \\
    \midrule
    94\% &99\% & 99\% & 97\% \\
    \bottomrule
  \end{tabular}
  \caption{Single-layer classifier accuracy achieved by different autoencoders on MNIST dataset.}
  \label{tab:classification}
\end{table}

\begin{table}
  \centering
  \begin{tabular}{@{}lcccc@{}}
    \toprule
    Baseline  & VAE & ACAI & ours \\
    \midrule
    54\%  &83\% & 96\% & 80\% \\
    \bottomrule
  \end{tabular}
  \caption{ Clustering accuracy for using K-Means on the latent space of different autoencoders on MNIST dataset}
  \label{tab:clustering}
\end{table}

\section{Conclusion}
Recently, there has been a great deal of interest in generative AI models such as ChatGPT and Stable Diffusion. However, these models require a large amount of training data. This paper focuses on generative AI for image generation, where training data is scarce. To address this issue, we proposed a path energy term to regularize the autoencoder, resulting in smooth and realistic interpolation results. Our experiments show that this idea has the potential to combine a robust mathematical model with neural network training. The main challenge in our training is solving linear systems in the path energy term. This could be alleviated by using other advanced techniques such as normalizing flow.
{
    \small
    \bibliographystyle{ieeenat_fullname}
    \bibliography{main}
}

% WARNING: do not forget to delete the supplementary pages from your submission 
\clearpage
\setcounter{page}{1}
\maketitlesupplementary

Details and source code related to the methodologies applied in this study are available in the GitHub repository: \url{https://github.com/xue1993/Dynamic-optimal-transport}.

\section{the Path Energy Term }
\label{sec:uotandobs}
In this section, we present the specifics of the path energy term when dealing with obstacles and unbalanced optimal transport.

For the path energy term based on original optima transport, we have a closed form as blow
\begin{equation*}
\begin{aligned}
J(\rho) = &\sum_{t=0}^{T-1} b_t^T \left(\mathbf{\nabla \cdot \operatorname{Diag}(\Q_t)}^{-1} \mathbf{\nabla \cdot ^{\top} }\right)^{-1} b_t.
\end{aligned}
\label{equ:quadraticmomentum}
\end{equation*}
Note that there is linear system to solve at each time scale.

When there are obstacles in the environment, we delete the momentum at that positions and solve reduced dimensional linear systems.

For the unbalanced optimal transport case, after solving the KKT system, we have 
\begin{equation*}
\begin{aligned}
J(\rho) = &\sum_{t=0}^{T-1} b_t^T \left(\mathbf{ [ \nabla \cdot \,\, I] \operatorname{Diag}(\Tilde{\Q}_t)}^{-1} \mathbf{[\nabla \cdot \,\, I] ^{\top} }\right)^{-1} b_t,
\end{aligned}
\label{equ:quadraticmomentum}
\end{equation*}
where $\mathbf{I}$ is the identity matrix,
$\Tilde{\Q}_t$ is a extended weight including the weight $\Q_t$ for momentum and the weight for source term defined as follows:
$$
\Q_t^c = \frac{\tau}{\rho_{t,i,j}}.
$$

\section{Additional expriements}
In this section, we present the results of experiments which are not the primary focus of our paper.

\subsection{Boundary Condition}
\label{boundarycondition}
In the dynamic optimal transport model, the continuity equation, a PDE constraint, is of great importance. The boundary condition of the divergence operator is a key factor in the solution. \Cref{fig:boundary} shows a comparison of the results obtained using different boundary conditions. It is evident that the Dirichlet boundary condition yields the best result, while the mass vanishes to the boundaries under other conditions.
Therefore, in this paper, we have chosen to use the Dirichlet boundary condition as default.

\begin{figure}[ht]
    \centering
    \begin{subfigure}[b]{\linewidth}
        \centering
        \includegraphics[width=\linewidth]{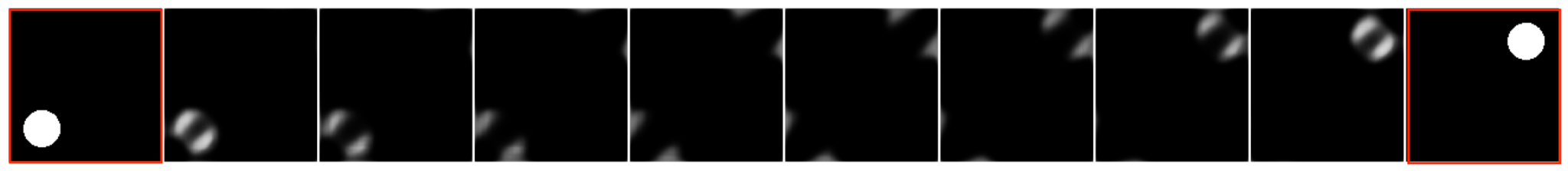}
        \caption{Dirichlet boundary condition}
        
        \label{fig:sub1}
    \end{subfigure}

    \begin{subfigure}[b]{\linewidth}
        \centering
        \includegraphics[width=\linewidth]{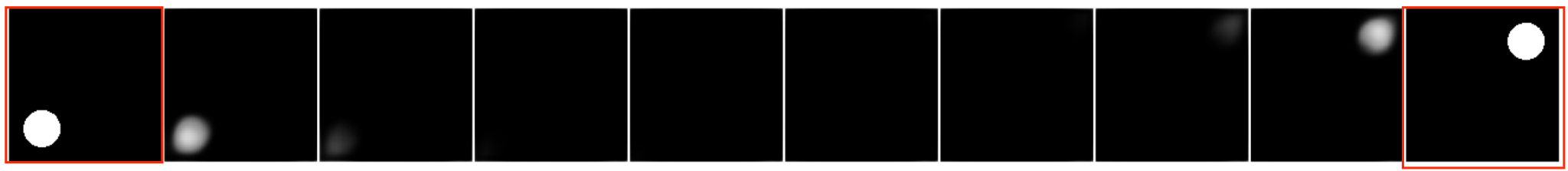}
        \caption{Periodic boundary condition.   }
        \label{fig:sub2}
        
    \end{subfigure}
    
    \begin{subfigure}[b]{\linewidth}
        \centering
        \includegraphics[width=\linewidth]{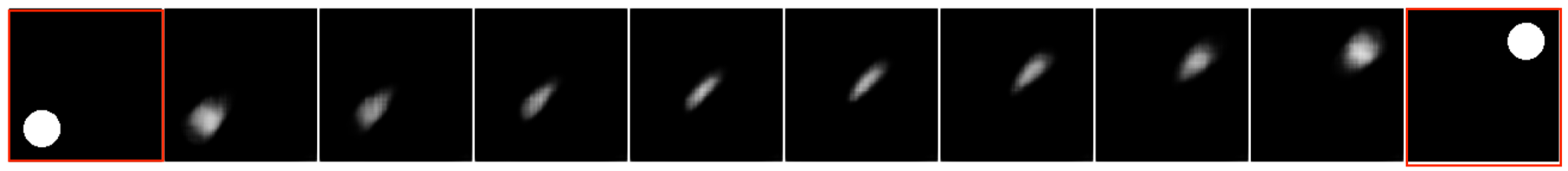}
        \caption{Neumann boundary condition  }
        \label{fig:sub2}
    \end{subfigure}

    \caption{The comparison result when trained with different boundary conditions}

    \label{fig:boundary}
    
\end{figure}

\subsection{Choice of T}
\label{Tchoice}
The discretion of time is of great importance to the interpolation result.
The selection of $T$ should be chosen depending on the maximum distance the pixels traverse. 
For example, if the pixels move a maximum of 10 pixels distance, then $T$ should be set to 10.

\subsection{the Comparison between OT and Unbalanced OT}
\label{otvsuot}
As we discussed earlier,
unbalanced OT is proposed to address the unbalanced mass between the source image and the target image.
It also makes it more suitable for shape-change dominated interpolation problems.
We show a comparison in \Cref{fig:comp}.
Our method based on the path energy term of UOT yields the best result, while OT and baseline autoencoders produce unsmooth or unrealistic transportation.

\begin{figure*}[ht]
    \centering
    \begin{subfigure}[b]{0.9\textwidth}
        \centering
        \includegraphics[width=0.9\textwidth]{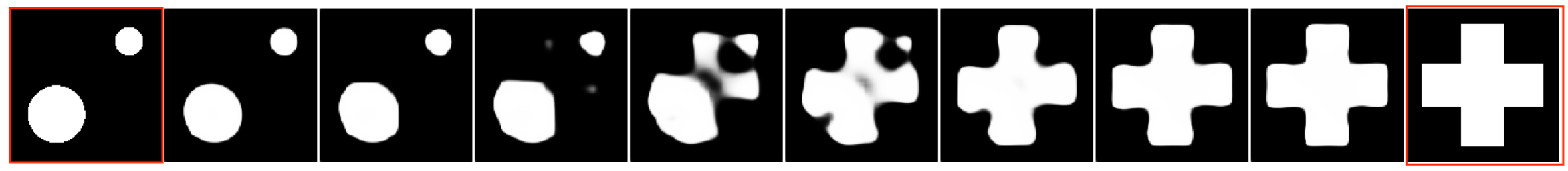}
        \caption{Baseline autoencoder.}
        
        \label{fig:sub1}
    \end{subfigure}

    \begin{subfigure}[b]{0.9\textwidth}
        \centering
        \includegraphics[width=0.9\textwidth]{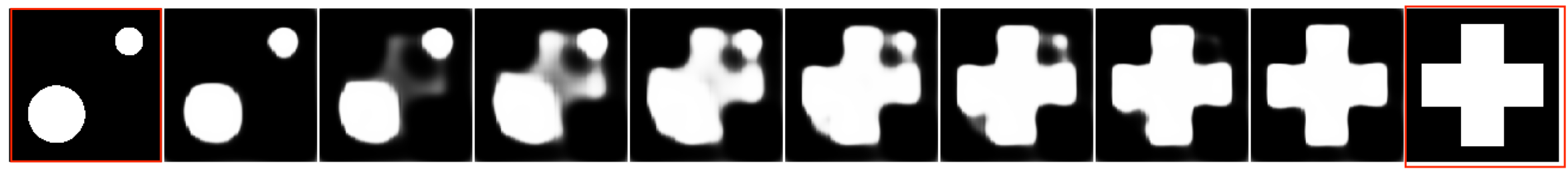}
        \caption{Our method with path energy term based on OT.   }
        \label{fig:sub2}
    \end{subfigure}

        \begin{subfigure}[b]{0.9\textwidth}
        \centering
        \includegraphics[width=0.9\textwidth]{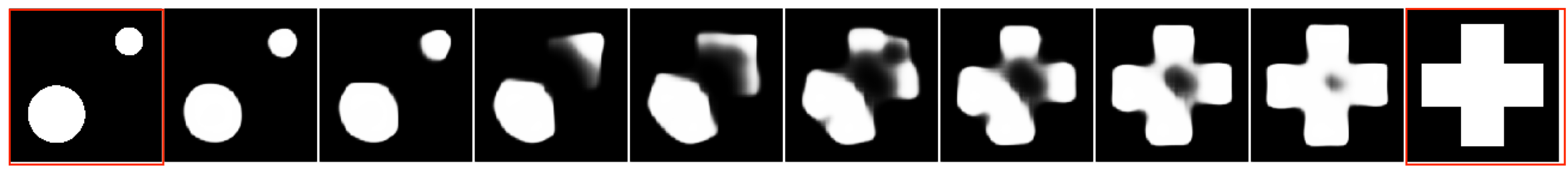}
        \caption{Our method with path energy term based on UOT.   }
        \label{fig:sub2}
    \end{subfigure}
\caption{Comparative analysis of image sequence interpolation. 
Given the images marked in red, the interpolated images are generated by decoding a linear combination of their latent codes after training.  }
    
    \label{fig:comp}
    
\end{figure*}

\subsection{Barycenter Problem}
In this section, we show the Barycenter results using other state of art methods, including MMOT\cite{zhou2022efficient} and convolutional Wasserstein barycenters in POT package\cite{flamary2021pot}.
Our result is shown in Figure 4 in section 4.
We can see that the result from POT is blurred, and the result of MMOT is very sharp but complicated.
Among them, our barycenter result is very fluid and smooth.

\begin{figure*}
  \centering
  \begin{subfigure}{0.5\linewidth}
  \includegraphics[width=0.85\textwidth]{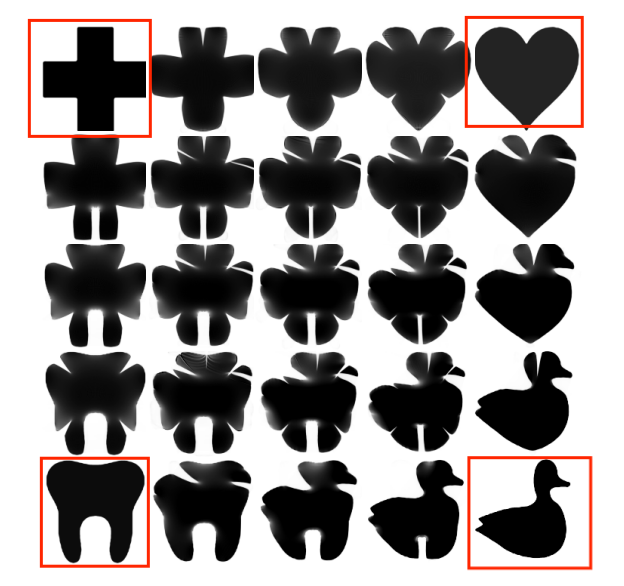}
    \caption{MMOT}
    \label{fig:short-a}
  \end{subfigure}
  \hfill
  \begin{subfigure}{0.45\linewidth}
  \includegraphics[width=0.85\textwidth]{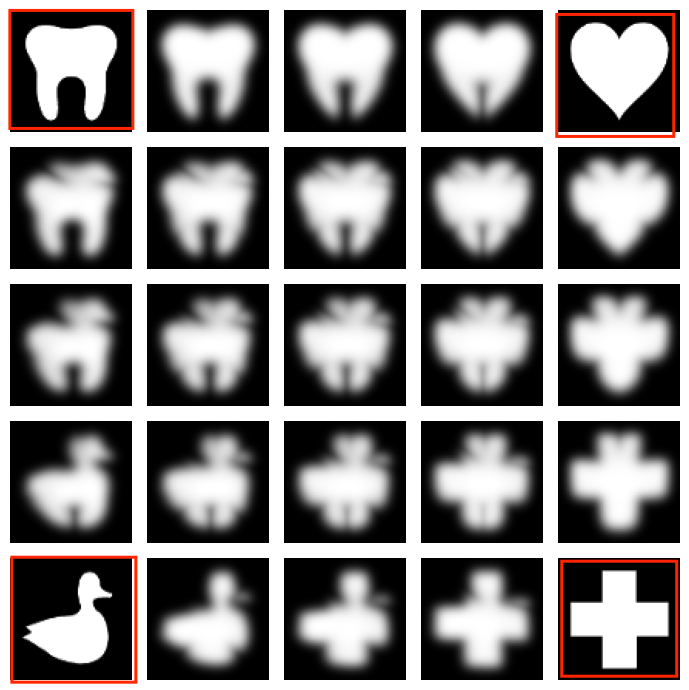}
    \caption{POT}
    \label{fig:short-b}
  \end{subfigure}
  \caption{Barycenter results using other methods.  }
  \label{fig:short}
\end{figure*}

\subsection{Auxiliary Training Data}
\label{auxiliary}
\label{sec:auxi}

One challenge in solving dynamic OT is the computation cost,
and this section explores whether the neural network learning nature of our method can bring benefits to finding the geodesic path.
To investigate this question, we design a simple but instructive comparison experiment.
Specifically,
we use the image of a circle in the upper left corner as the source image $x_1$ and use the image of a circle in the lower right corner as the target image $x_2$, shown in Figure \ref{fig:auxi}.
The aim is to generate the geodesic path between $x_1$ and $x_2$. 
We also create images with a center circle at every point in time, with a total of 1936 images.
The first experiment uses $x_1$ and $x_2$ as the only training data,
while the second experiment uses all training data with $1,2$ as the only choice of $i,j$ pair in the regularization term. 
Thus, the second experiment has 1934 auxiliary data.
The batch size is 500.
In both examples,
the blocks in both encoder and decoder are repeated 4 times, resulting in a smaller latent dimension $16\times2\times2$.
The results (Figure \ref{fig:auxi}) indicate that the auxiliary data significantly improve the interpolation quality. 
Without auxiliary data, the interpolation is blurrier at epoch 200 compared to the result at epoch 50 with auxiliary data. 
Additionally, the final interpolation result is less accurate without auxiliary data. 
Therefore, the inclusion of auxiliary data improves the autoencoder's efficiency and accuracy in generating the geodesic path.

\begin{figure*}
%\right
\begin{flushright}

\rotatebox[origin=c]{0}{Epoch 200}\quad
{\includegraphics[width=0.85\textwidth,valign=c]{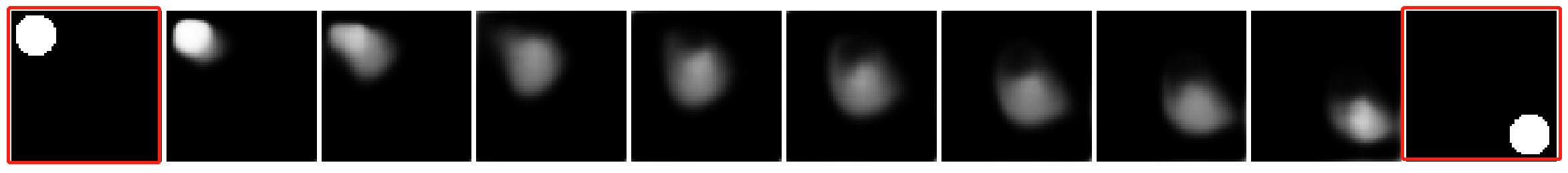}}

\rotatebox[origin=c]{0}{ 1000}\quad
{\includegraphics[width=0.85\textwidth,valign=c]{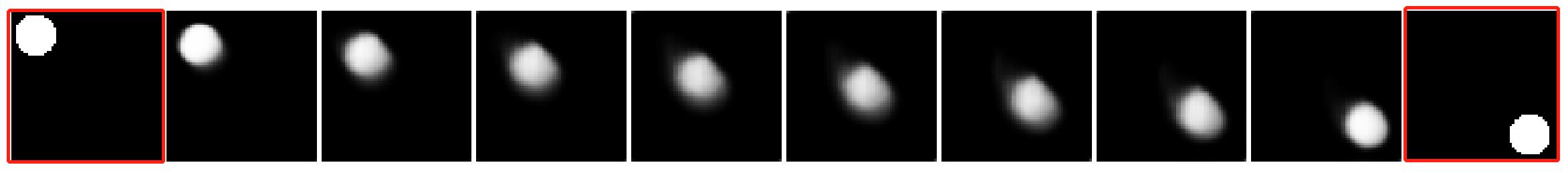}}

\begin{center} (a)  \end{center}

\rotatebox[origin=c]{0}{Epoch 50\,}\quad
{\includegraphics[width=0.85\textwidth,valign=c]{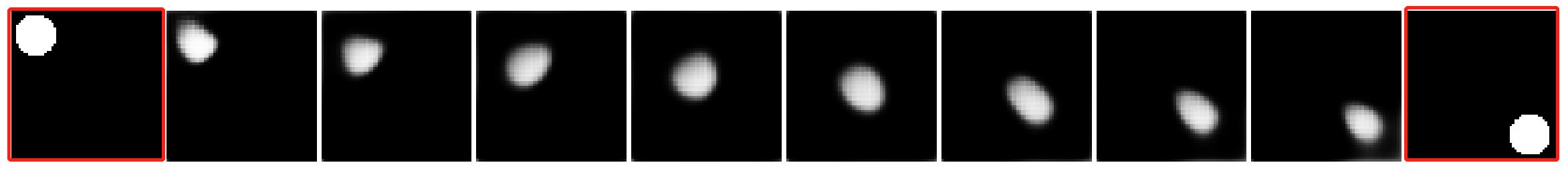}}

\rotatebox[origin=c]{0}{ 250}\quad
{\includegraphics[width=0.85\textwidth,valign=c]{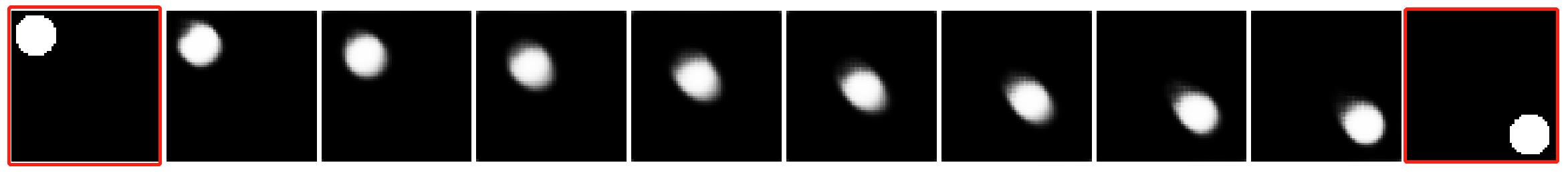}}

\begin{center} (b)\end{center}
\end{flushright}

\caption{Comparison of interpolation results at different training epochs. (a) Results with only two training data points at epochs 200 and 1000. (b) Results with auxiliary data at epochs 50 and 250. Detailed descriptions of the training can be found in Section \ref{sec:auxi}.}

\label{fig:auxi}
\end{figure*}

\end{document}